\documentclass[12pt]{amsart}

\usepackage{amsthm,amsmath,amsfonts,amssymb,mathrsfs}
\usepackage[colorlinks]{hyperref}
\usepackage{xcolor}
\usepackage{enumerate}

\allowdisplaybreaks

\theoremstyle{plain}
\newtheorem{thm}{Theorem}[section]

\newtheorem{lem}[thm]{Lemma}
\newtheorem{prop}[thm]{Proposition}
\newtheorem{cor}[thm]{Corollary}

\newtheorem{exam}[thm]{Example}

\title[Pego theorem on compact groups]{Pego theorem on compact groups}

\author{Manoj Kumar}
\address{Department of Mathematics \endgraf Indian Institute of Science Bangalore \endgraf Bangalore - 560012 \endgraf India}
\email{manojk9t3@gmail.com}

\begin{document}
	
\begin{abstract}
	The Pego theorem characterizes the precompact subsets of the square-integrable functions on $\mathbb{R}^n$ via the Fourier transform. We prove the analogue of the Pego theorem on compact groups (not necessarily abelian).
\end{abstract}



\maketitle

\section{Introduction}
Characterizing precompact subsets is one of the classical topics in function space theory. It is well known that the Arzel\`a-Ascoli theorem characterizes a precompact subset of space of continuous functions over compact Hausdorff space. Further, the celebrated Riesz-Kolmogorov theorem provides a characterization of precompact subsets of $L^p(\mathbb{R}^n).$ We refer \cite{HH} for a historical account of it. Weil \cite[Pg. 52]{W} extended it to the Lebesgue spaces over locally compact groups. See \cite{GP19} for its extension to the Banach function spaces over locally compact groups. 

In 1985, Pego \cite{P} used the Riesz-Kolmogorov theorem to find a  characterization of precompact subsets of $L^2(\mathbb{R}^n)$ via certain decay of the Fourier transform. It states as follows:

\begin{thm}\label{Pe}{\cite[Theorems 2 and 3]{P}} \mbox{ }
    Let $K$ be a bounded subset of $L^2(\mathbb{R}^n).$ Then, the following are equivalent:
    \begin{enumerate}[(i)]
        \item $K$ is precompact.
        \item $\int_{|x|>r}|f(x)|^2\,dx\rightarrow 0 \mbox{ and } \int_{|\xi|>r}|\widehat{f}(\xi)|^2\,d\xi\rightarrow 0$ as $r\rightarrow\infty,$ both uniformly for $f$ in $K.$
        \item $\int_{\mathbb{R}^n} |f(x+y)-f(x)|^2\,dx\rightarrow 0$ as $y\rightarrow 0,$ and $\int_{\mathbb{R}^n} |\widehat{f}(\xi+\omega)-\widehat{f}(\xi)|^2\,d\xi\rightarrow 0$ as $\omega\rightarrow 0,$ both uniformly for $f$ in $K.$
    \end{enumerate}
\end{thm}
An application of this theorem to information theory has also been provided in \cite{P}. 

Pego type theorem has also been studied via the short-time Fourier and wavelet transforms \cite{DFG}, the Laplace transform \cite{K20} and the Laguerre and Hankel transforms \cite{hor}. The Pego theorem has been extended to the locally compact abelian groups with some technical assumptions \cite{G}. Using the Pontryagin duality and the Arzela-Ascoli theorem, the authors in \cite{ GK16} showed that the technical assumptions are redundant. For $L^1$-space analogue of Pego type theorem over locally compact abelian groups, see \cite{K}. 

In Section \ref{FA}, we present preliminaries on compact groups. In Section \ref{PT}, using Weil's compactness theorem, we extend the Theorem \ref{Pe} to compact groups (not necessarily abelian); see Theorem \ref{Pego}. 

\section{Fourier analysis on compact groups}\label{FA}

Let $G$ be a compact Hausdorff group. Assume that $m_G$ denotes the normalized positive Haar measure on $G.$ Let $L^p(G)$ denote the $p$th Lebesgue space w.r.t. the measure $m_G.$ The norm on the space $L^p(G)$ is denoted by $\|\cdot\|_p.$

We denote by $\widehat{G}$ the space consisting of all irreducible unitary representations of $G$ up to the unitary equivalence. The set $\widehat{G}$ is known as the unitary dual of $G$ and is equipped with the discrete topology. Note that the representation space $\mathcal{H}_\pi$ of $\pi\in\widehat{G}$ is a complex Hilbert space and finite-dimensional. Denote by $d_\pi$ the dimension of $\mathcal{H}_\pi.$ 

Let $\Lambda\subset\widehat{G}.$ Assume that $\{(X_\pi,\|\cdot\|_\pi): \pi\in\Lambda\}$ is a family of Banach spaces. For $1\leq p<\infty,$ we denote by $\ell^p\mbox{-}\underset{\pi\in\wedge}{\oplus}X_\pi$ the Banach space $$\left\{(x_\pi)\in\underset{\pi\in\wedge}{\Pi}X_\pi:\underset{\pi\in\wedge}{\sum}d_\pi\|x_\pi\|_\pi^p<\infty\right\}$$ endowed with the norm $\|(x_\pi)\|_{\ell^p\mbox{-}\underset{\pi\in\wedge}{\oplus}X_\pi}:=\left(\underset{\pi\in\wedge}{\sum}d_\pi\|x_\pi\|_\pi^p\right)^{1/p}.$ Denote by $\ell^\infty\mbox{-}\underset{\pi\in\wedge}{\oplus}X_\pi$ the Banach space $$\left\{(x_\pi)\in\underset{\pi\in\wedge}{\Pi}X_\pi:\underset{\pi\in\wedge}{\sup}\|x_\pi\|_\pi<\infty\right\}$$ endowed with the norm $\|(x_\pi)\|_{\ell^\infty\mbox{-}\underset{\pi\in\wedge}{\oplus}X_\pi}:=\underset{\pi\in\wedge}{\sup}\|x_\pi\|_\pi.$ Similarly, denote by $c_0\mbox{-}\underset{\pi\in\wedge}{\oplus}X_\pi$ the space consisting of $(x_\pi)$ from $\ell^\infty\mbox{-}\underset{\pi\in\wedge}{\oplus}X_\pi$ such that $x_\pi\rightarrow 0$ as $\pi\rightarrow\infty,$ i.e., for any given $\epsilon>0$ there exists a finite set $\Lambda_\epsilon\subset\Lambda$ such that $\|x_\pi\|_\pi<\epsilon$ for all $\pi\in\Lambda\setminus\Lambda_\epsilon.$ Note that $c_0\mbox{-}\underset{\pi\in\wedge}{\oplus}X_\pi$ is a closed subspace of $\ell^\infty\mbox{-}\underset{\pi\in\wedge}{\oplus}X_\pi.$

For $1\leq p<\infty,$ let $\mathcal{B}_p(\mathcal{H}_\pi)$ denote the space of all bounded linear operators $T$ on $\mathcal{H}_\pi$ such that $\|T\|_{\mathcal{B}_p(\mathcal{H}_\pi)}:=(\mathrm{tr}(|T|^p))^{1/p}<\infty.$ The space $\mathcal{B}_2(\mathcal{H}_\pi)$ is called the space of the Hilbert-Schmidt operators on the Hilbert space $\mathcal{H}_\pi.$ The space $\mathcal{B}_2(\mathcal{H}_\pi)$ is Hilbert space endowed with the inner product $$\langle T,S \rangle_{\mathcal{B}_2(\mathcal{H}_\pi)}:=\mathrm{tr}(TS^*).$$ Further, let $\mathcal{B}(\mathcal{H}_\pi)$ denote the space consisting of all bounded linear operators on $\mathcal{H}_\pi$ endowed with the operator norm. 

Let $f\in L^1(G).$ The Fourier transform of $f$ is defined by $$\widehat{f}(\pi)=\int_Gf(t)\pi(t)^* \,dm_G(t),\ \pi\in\widehat{G}.$$ The Fourier transform operator $f\mapsto\widehat{f}$ from $L^1(G)$ into $\ell^\infty\mbox{-}\underset{\pi\in\widehat{G}}{\oplus}\mathcal{B}(\mathcal{H}_\pi)$ is injective and bounded. By the Riemann-Lebesgue Lemma, we know that $\widehat{f}\in c_0\mbox{-}\underset{\pi\in\widehat{G}}{\oplus}\mathcal{B}(\mathcal{H}_\pi).$ The convolution of $f,g\in L^1(G)$ is given by $$f*g(x)=\int_G f(xy^{-1})g(y)\,dm_G(y).$$ Then, $\widehat{f*g}(\pi)=\widehat{g}(\pi)\widehat{f}(\pi),\ \pi\in\widehat{G}.$  For $y\in G$, the right translation $R_y$ of $f\in L^p(G)$ is given by $R_y(f)(x)=f(xy),\ x\in G.$ Then, $\widehat{R_yf}(\pi)=\pi(y)\widehat{f}(\pi),\ \pi\in\widehat{G}.$

For more information on compact groups, we refer to \cite{F, HR}.

Throughout the paper, $G$ will denote a compact Hausdorff group (not necessarily abelian). The identity of $G$ is denoted by $e.$ We will denote by $I_{d_\pi}$ the $d_\pi\times d_\pi$ identity matrix.

\section{Pego theorem on compact groups}\label{PT}

In this section, we discuss the characterization of precompact subsets of square-integrable functions on $G$ in terms of the Fourier transform. We need the following definitions.

Let $K\subset L^p(G).$ Define $\widehat{K}:=\{\widehat{f}:f\in L^p(G)\}.$ $K$ is said to be {\it uniformly $L^p(G)$-equicontinuous} if for any given $\epsilon>0$ there exists an open neighbourhood $O$ of $e$ such that $$\|R_yf-f\|_p<\epsilon,\ f\in K \mbox{ and } y\in O.$$ 

Let $F\subset\ell^p\mbox{-}\underset{\pi\in\widehat{G}}{\oplus}\mathcal{B}_p(\mathcal{H}_\pi).$ $F$ is said to have {\it uniform $\ell^p\mbox{-}\underset{\pi\in\widehat{G}}{\oplus}\mathcal{B}_p(\mathcal{H}_\pi)$-decay} if for any given $\epsilon>0$ there exists a finite set $A\subset\widehat{G}$ such that $$\|\phi\|_{\ell^p\mbox{-}\underset{\pi\in\widehat{G}\setminus A}{\oplus}\mathcal{B}_p(\mathcal{H}_\pi)}<\epsilon,\ \phi\in F.$$

Let us begin with some important lemmas.

\begin{lem}\label{decay}
	Let $K\subset L^p(G),$ where $p\in[1,2].$ If $K$ is uniformly $L^p(G)$-equicontinuous then $\widehat{K}$ has uniform $\ell^{p'}\mbox{-}\underset{\pi\in\widehat{G}}{\oplus}\mathcal{B}_{p'}(\mathcal{H}_\pi)$-decay.
\end{lem}

\begin{proof}
	Let $(e_U)_{U\in\Lambda}$ be a Dirac net on $G$; see \cite[Pg. 28]{D}. By the Riemann-Lebesgue Lemma \cite[Theorem 28.40]{HR},  $\widehat{e_U}\in c_0\mbox{-}\underset{\pi\in\widehat{G}}{\oplus}\mathcal{B}(\mathcal{H}_\pi).$ Then, there exists a finite set $A\subset\widehat{G}$ such that $$\|\widehat{e_U}(\pi)\|_{\mathcal{B}(\mathcal{H}_\pi)}\leq\frac{1}{2},\ \pi\in\widehat{G}\setminus A.$$
	
	Let $f\in K.$ We denote by $\widehat{e_U}\widehat{f}$ the pointwise product of $\widehat{e_U}$ and $\widehat{f}.$ Now, 
	\begin{align*}
		\|\widehat{f}\|_{\ell^{p'}\mbox{-}\underset{\pi\in\widehat{G}\setminus A}{\oplus}\mathcal{B}_{p'}(\mathcal{H}_\pi)}\leq&\|\widehat{f}-\widehat{e_U}\widehat{f}\|_{\ell^{p'}\mbox{-}\underset{\pi\in\widehat{G}\setminus A}{\oplus}\mathcal{B}_{p'}(\mathcal{H}_\pi)}+\|\widehat{e_U}\widehat{f}\|_{\ell^{p'}\mbox{-}\underset{\pi\in\widehat{G}\setminus A}{\oplus}\mathcal{B}_{p'}(\mathcal{H}_\pi)}\\\leq&\|\widehat{f-f*e_U}\|_{\ell^{p'}\mbox{-}\underset{\pi\in\widehat{G}\setminus A}{\oplus}\mathcal{B}_{p'}(\mathcal{H}_\pi)}+\|\widehat{f}\|_{\ell^{p'}\mbox{-}\underset{\pi\in\widehat{G}\setminus A}{\oplus}\mathcal{B}_{p'}(\mathcal{H}_\pi)}\sup_{\pi\in\widehat{G}\setminus A}\|\widehat{e_U}(\pi)\|_{\mathcal{B}(\mathcal{H}_\pi)}\\\leq&\|\widehat{f-f*e_U}\|_{\ell^{p'}\mbox{-}\underset{\pi\in\widehat{G}\setminus A}{\oplus}\mathcal{B}_{p'}(\mathcal{H}_\pi)}+\frac{1}{2}\|\widehat{f}\|_{\ell^{p'}\mbox{-}\underset{\pi\in\widehat{G}\setminus A}{\oplus}\mathcal{B}_{p'}(\mathcal{H}_\pi)}.
	\end{align*}
	
	Then, applying the Hausdorff-Young inequality \cite[Theorem 31.22]{HR}, we get
	\begin{align*}
		\|\widehat{f}\|_{\ell^{p'}\mbox{-}\underset{\pi\in\widehat{G}\setminus A}{\oplus}\mathcal{B}_{p'}(\mathcal{H}_\pi)}\leq&2\|\widehat{f-f*e_U}\|_{\ell^{p'}\mbox{-}\underset{\pi\in\widehat{G}}{\oplus}\mathcal{B}_{p'}(\mathcal{H}_\pi)}\\\leq&2\|f-f*e_U\|_p\\=&2\left(\int_G|f(x)-f*e_U(x)|^p\,dm_G(x)\right)^{1/p}\\=&2\left(\int_G\left|\int_G(f(x)-f(xy^{-1}))e_U(y)\,dm_G(y)\right|^p\,dm_G(x)\right)^{1/p}.
	\end{align*}
	Therefore, using the Minkowski integral inequality, we obtain
	\begin{align*}
	\|\widehat{f}\|_{\ell^{p'}\mbox{-}\underset{\pi\in\widehat{G}\setminus A}{\oplus}\mathcal{B}_{p'}(\mathcal{H}_\pi)}\leq&2\int_G\left(\int_G|f(x)-f(xy^{-1})|^p\,dm_G(x)\right)^{1/p} e_U(y)\,dm_G(y)\\\leq&2\sup_{y\in U}\left(\int_G|f(x)-f(xy^{-1})|^p\,dm_G(x)\right)^{1/p}.
	\end{align*}

	Let $\epsilon>0.$ Since $K$ is uniformly $L^p(G)$-equicontinuous, there exists an open neighbourhood $O$ of $e$ such that $$\|R_yf-f\|_p<\frac{\epsilon}{2},\ f\in K \mbox{ and } y\in O.$$ By \cite[Proposition 2.1 (b)]{F}, we get that there exists $U\in\Lambda$ such that $$\|R_yf-f\|_p<\frac{\epsilon}{2},\ f\in K \mbox{ and } y\in U.$$ Hence, 
	\begin{align*}
		\|\widehat{f}\|_{\ell^{p'}\mbox{-}\underset{\pi\in\widehat{G}\setminus A}{\oplus}\mathcal{B}_{p'}(\mathcal{H}_\pi)}<&\epsilon,\ f\in K.\qedhere
	\end{align*}
\end{proof}

\begin{lem}\label{equicont}
	Let $K\subset L^{p'}(G),$ where $p\in[1,2].$ If $\widehat{K}$ has uniform $\ell^p\mbox{-}\underset{\pi\in\widehat{G}}{\oplus}\mathcal{B}_p(\mathcal{H}_\pi)$-decay then $K$ is uniformly $L^{p'}(G)$-equicontinuous.
\end{lem}

\begin{proof}Let $\epsilon>0.$ Since $\widehat{K}$ has uniform $\ell^p\mbox{-}\underset{\pi\in\widehat{G}}{\oplus}\mathcal{B}_p(\mathcal{H}_\pi)$-decay, there exists a finite set $A\subset\widehat{G}$ such that $$\|\widehat{f}\|_{\ell^p\mbox{-}\underset{\pi\in\widehat{G}\setminus A}{\oplus}\mathcal{B}_p(\mathcal{H}_\pi)}<\frac{\epsilon}{4},\ f\in K.$$
	
	Let $f\in K$ and $y\in G.$ Then, applying \cite[Corollary 31.25]{HR}, we obtain 
	\begin{align*}
	\|R_yf-f\|_{p'}\leq&\|\widehat{R_yf-f}\|_{\ell^p\mbox{-}\underset{\pi\in\widehat{G}}{\oplus}\mathcal{B}_p(\mathcal{H}_\pi)}\\=&\left(\sum_{\pi\in\widehat{G}}d_\pi\|\widehat{R_yf}(\pi)-\widehat{f}(\pi)\|_{\mathcal{B}_p(\mathcal{H}_\pi)}^p\right)^{1/p}\\\leq&\left(\sum_{\pi\in A}d_\pi\|\pi(y)\widehat{f}(\pi)-\widehat{f}(\pi)\|_{\mathcal{B}_p(\mathcal{H}_\pi)}^p\right)^{1/p}\\&+\left(\sum_{\pi\in\widehat{G}\setminus A}d_\pi\|\pi(y)\widehat{f}(\pi)-\widehat{f}(\pi)\|_{\mathcal{B}_p(\mathcal{H}_\pi)}^p\right)^{1/p}\\\leq&\sup_{\pi\in A}\|\pi(y)-I_{d_\pi}\|_{\mathcal{B}(\mathcal{H}_\pi)}\left(\sum_{\pi\in A}d_\pi\|\widehat{f}(\pi)\|_{\mathcal{B}_p(\mathcal{H}_\pi)}^p\right)^{1/p}\\&+\sup_{\pi\in \widehat{G}\setminus A}\|\pi(y)-I_{d_\pi}\|_{\mathcal{B}(\mathcal{H}_\pi)}\left(\sum_{\pi\in\widehat{G}\setminus A}d_\pi\|\widehat{f}(\pi)\|_{\mathcal{B}_p(\mathcal{H}_\pi)}^p\right)^{1/p}\\\leq&M\sup_{\pi\in A}\|\pi(y)-I_{d_\pi}\|_{\mathcal{B}(\mathcal{H}_\pi)}+\frac{\epsilon}{2},
	\end{align*}
	where $M$ is a positive number such that $\left(\sum_{\pi\in A}d_\pi\|\widehat{f}(\pi)\|_{\mathcal{B}_p(\mathcal{H}_\pi)}^p\right)^{1/p}\leq M.$
	
	Let $\pi\in A.$ Using continuity of $\pi,$ we obtain that there exists a neighbourhood $O_\pi$ of $e$ such that $$\|\pi(y)-I_{d_\pi}\|_{\mathcal{B}(\mathcal{H}_\pi)}<\frac{\epsilon}{2M},\ y\in O_\pi.$$ Assume that $O=\cap_{\pi\in A}O_\pi.$ Then, $$\|\pi(y)-I_{d_\pi}\|_{\mathcal{B}(\mathcal{H}_\pi)}<\frac{\epsilon}{2M},\ \pi\in A \mbox{ and } y\in O.$$ Hence, \begin{align*}
		\|R_yf-f\|_{p'}<&\epsilon,\ f\in K \mbox{ and } y\in O. \qedhere
	\end{align*}
\end{proof}

The following corollary is a generalization of \cite[Theorem 1]{P} studied on $\mathbb{R}^n$, and \cite[Theorem 1]{G} and \cite[Lemma 2.5]{Feich} studied on locally compact abelian groups. This is also an improvement of the corresponding result on compact abelian groups in the sense that we do not assume boundedness of the subset of $L^2(G).$ 

\begin{cor}\label{CharcCor}
	Let $K\subset L^2(G).$ Then, $K$ is uniformly $L^2(G)$-equicontinuous if and only if $\widehat{K}$ has uniform $\ell^2\mbox{-}\underset{\pi\in\widehat{G}}{\oplus}\mathcal{B}_2(\mathcal{H}_\pi)$-decay. 
\end{cor}

\begin{proof}
    This is a direct consequence of Lemma \ref{decay} and Lemma \ref{equicont}.
\end{proof}

Now, we prove some propositions that we use to establish the Pego theorem on compact groups. 

\begin{prop}\label{decayProp}
	Let $K\subset L^2(G).$ If $K$ is precompact then $\widehat{K}$ has uniform $\ell^2\mbox{-}\underset{\pi\in\widehat{G}}{\oplus}\mathcal{B}_2(\mathcal{H}_\pi)$-decay.
\end{prop}

\begin{proof}
	By the Weil theorem \cite[Pg. 52]{W} (or see \cite[Theorem 3.3]{GP19}), $K$ is uniformly $L^2(G)$-equicontinuous. Then, applying Corollary \ref{CharcCor}, we get that $\widehat{K}$ has uniform $\ell^2\mbox{-}\underset{\pi\in\widehat{G}}{\oplus}\mathcal{B}_2(\mathcal{H}_\pi)$-decay.
\end{proof}

The following is the converse to the previous proposition.

\begin{prop}\label{precomProp}
	Let $K$ be a bounded subset of $L^2(G).$ If $\widehat{K}$ has uniform $\ell^2\mbox{-}\underset{\pi\in\widehat{G}}{\oplus}\mathcal{B}_2(\mathcal{H}_\pi)$-decay then $K$ is precompact.
\end{prop}

\begin{proof}
	Using Corollary \ref{CharcCor}, we have that $K$ is uniformly $L^2(G)$-equicontinuous. For any given $\epsilon>0$ we have that $$\sup_{f\in K}\|f\chi_{G\setminus G}\|_2=0<\epsilon.$$  Since $K$ is bounded, it follows by the Weil theorem \cite[Pg. 52]{W} (or see \cite[Theorem 3.1]{GP19}) that $K$ is precompact.
\end{proof}

Now, we present an analogue of the Pego theorem over compact groups. Combining Corollary \ref{CharcCor}, Proposition \ref{decayProp} and Proposition \ref{precomProp} gives the following theorem.
\begin{thm}\label{Pego}
    Let $K$ be a bounded subset of $L^2(G).$ Then, the following are equivalent:
    \begin{enumerate}[(i)]
        \item $K$ is precompact.
        \item $\widehat{K}$ has uniform $\ell^2\mbox{-}\underset{\pi\in\widehat{G}}{\oplus}\mathcal{B}_2(\mathcal{H}_\pi)$-decay.
        \item $K$ is uniformly $L^2(G)$-equicontinuous.
    \end{enumerate}
\end{thm}

The following gives an example of a set $K\subset L^2(G)$ which is not precompact but $K$ is uniformly $L^2(G)$-equicontinuous and $\widehat{K}$ has uniform $\ell^2\mbox{-}\underset{\pi\in\widehat{G}}{\oplus}\mathcal{B}_2(\mathcal{H}_\pi)$-decay.

\begin{exam}
	Consider the set $K=\{n\chi_G:n\in\mathbb{N}\}\subset L^2(G)$ as given in \cite[Example 4.2]{GP19}. Since $K$ consists of only constant functions, it is clear that $K$ is uniformly $L^2(G)$-equicontinuous. By Corollary \ref{CharcCor}, $\widehat{K}$ has uniform $\ell^2\mbox{-}\underset{\pi\in\widehat{G}}{\oplus}\mathcal{B}_2(\mathcal{H}_\pi)$-decay. Since $K$ is not bounded, $K$ is not precompact.
\end{exam}

Now, with the help of our main result Theorem \ref{Pego}, we show that certain subsets of $L^2(G)$ are precompact.

\begin{exam}\mbox{ }
\begin{enumerate}[{\bf (i)}]
    \item Let $r\in\mathbb{R}.$ Consider the set $K=\{\frac{r}{n}\chi_G:n\in\mathbb{N}\}\subset L^2(G).$ Since $\{\frac{r}{n}:n\in\mathbb{N}\}$ is bounded and $K$ consists of only constant functions, it follows that $K$ is bounded and uniformly $L^2(G)$-equicontinuous.  Therefore, by Theorem \ref{Pego}, $K$ is precompact.

    \item Let $A$ be a finite subset of $\widehat{G}.$ Assume that $K$ is a bounded subset of the linear span of the set consisting of matrix entries \cite[Pg. 139]{F} of elements in $A.$ Since the matrix entries are bounded functions in $L^2(G),$ $K$ is bounded subset of $L^2(G).$ For $f\in K,$ using the Schur orthogonality relations \cite[Theorem 5.8]{F} we obtain that  $$\|\widehat{f}\|_{\ell^2\mbox{-}\underset{\pi\in\widehat{G}\setminus A}{\oplus}\mathcal{B}_2(\mathcal{H}_\pi)}=0.$$ Thus, $\widehat{K}$ has uniform $\ell^2\mbox{-}\underset{\pi\in\widehat{G}}{\oplus}\mathcal{B}_2(\mathcal{H}_\pi)$-decay. Hence, by Theorem \ref{Pego}, $K$ is precompact. In particular, the convex hull of the set consisting of matrix entries of elements in $A$ is precompact.
    
\end{enumerate}
\end{exam}

\section*{Acknowledgment}
	The author is supported by the NBHM post-doctoral fellowship with Ref. number: 0204/3/2021/R\&D-II/7356 from the Department of Atomic Energy (DAE), Government of India. The author is grateful to Sundaram Thangavelu for his useful comments.


\begin{thebibliography}{aaaa}
\bibitem{D} \textsc{A. Deitmer} and \textsc{S. Echterhoff}, {\it Principles of Harmonic Analysis}, Springer, 2009.
\bibitem{DFG} \textsc{M. D\"orfler}, \textsc{H. G. Feichtinger} and \textsc{K. Gr\"ochenig}, Compactness criteria in function spaces, {\it Colloq. Math.}, 94 (2002) 37–50.
\bibitem{Feich} \textsc{H.G. Feichtinger}, Compactness in translation invariant Banach spaces of distributions and compact multipliers, {\it J. Math. Anal. Appl.}, 102 (1984) 289–327.
\bibitem{F} \textsc{G. B. Folland}, {\it A Course in Abstract Harmonic Analysis} (2nd ed.), CRC Press, Boca Raton, 2016.
\bibitem{G} \textsc{P. Górka}, Pego theorem on locally compact abelian groups, {\it J. Algebra Appl.}, 13(4) (2014) 1350143 (5 pages).
\bibitem{GK16} \textsc{P. Górka} and \textsc{T. Kostrzewa}, Pego everywhere, {\it J. Algebra Appl.}, 15 (2016), no. 4, 1650074, 3 pp.
\bibitem{GP19} \textsc{P. Górka} and \textsc{P. Po\'{s}piech}, Banach function spaces on locally compact groups, {\it Ann. Funct. Anal.} 10 (2019), 460-471.
\bibitem{HH} \textsc{H. Hanche-Olsen} and \textsc{H. Holden}, The Kolmogorov-Riesz compactness theorem, {\it Expo. Math.}, 28 (2010), no. 4, 385-394.
\bibitem{HR} \textsc{E. Hewitt}, and \textsc{K.A. Ross}, {\it Abstract harmonic analysis}, Volume II. - Grundlehren Math. Wiss., Springer-Verlag, New York-Berlin, 1970.
\bibitem {hor} \textsc{\'A. P. Horv\'ath}, Compactness criteria via Laguerre and Hankel transformations, {\it J. Math. Anal. Appl.},  507 (2022) 125852.
\bibitem{K20} \textsc{M. Krukowski}, Characterizing compact families via Laplace transform, {\it Ann. Acad. Sci. Fenn.}, Math. 45 (2020) 991–1002.
\bibitem{K} \textsc{M. Krukowski}, How Arzel\`a and Ascoli would have proved Pego theorem for $L^1(G)$ (if they lived in the $21^{\text{st}}$ century)?, arXiv:2006.12130.
\bibitem{P} \textsc{R. L. Pego}, Compactness in $L^2$ and the Fourier transform, {\it Proc. Amer. Math. Soc.}, 95 (1985) 252–254.
\bibitem{W} \textsc{A. Weil}, {\it L'int\'{e}gration dans les groupes topologiques et ses applications}, Actual. Sci. Ind. 869, Hermann, Paris, 1940.
\end{thebibliography}
\end{document}